\documentclass[12pt,a4paper]{article}
\let\mylabel\label

\usepackage{color}
\usepackage{amsmath,amssymb}
\usepackage{bm}
\bmdefine{\aaa}{a}
\bmdefine{\bbb}{b}
\bmdefine{\ccc}{c}
\bmdefine{\mmm}{m}
\bmdefine{\ppp}{p}
\bmdefine{\qqq}{q}
\bmdefine{\ttt}{t}
\bmdefine{\uuu}{u}
\bmdefine{\vvv}{v}
\bmdefine{\www}{w}
\bmdefine{\eee}{e}
\bmdefine{\xxx}{x}
\bmdefine{\yyy}{y}
\bmdefine{\zzz}{z}
\bmdefine{\zerovec}{0}
\bmdefine{\onevec}{1}

\newcommand{\CCC}{\mathbb{C}}
\newcommand{\RRR}{\mathbb{R}}
\newcommand{\FFF}{\mathbb{F}}
\newcommand{\KKK}{\mathbb{K}}

\newcommand{\define}{\mathrel{:=}}

\newcommand{\rank}{\mathrm{rank}}

\newcommand{\diag}{{\mathrm{Diag}}}

\newcommand{\glin}{{\mathrm{GL}}}

\newcommand{\image}{{\mathrm{Im}}}
\newcommand{\dom}{{\mathrm{dom}}}

\newcommand{\trank}{{\mathrm{typical\_rank}}}
\newcommand{\grank}{{\mathrm{generic\_rank}}}
\newcommand{\crank}{{\mathrm{column\_rank}}}
\newcommand{\rrank}{{\mathrm{row\_rank}}}

\newcommand{\afr}{{absolutely fullrank}}

\newcommand{\afcr}{{absolutely full column rank}}
\newcommand{\Afcr}{{Absolutely full column rank}}

\newcommand{\ans}{{absolutely nonsingular}}
\newcommand{\hr}{{Hurwitz-Radon}}

\newcommand{\transpose}{^\top}

\newcounter{condition}
\numberwithin{condition}{section}

\numberwithin{equation}{section}

\newtheorem{thm}{Theorem}

\newtheorem{example}[thm]{Example}
\newtheorem{lemma}[thm]{Lemma}
\newtheorem{cor}[thm]{Corollary}
\newtheorem{definition}[thm]{Definition}
\newtheorem{prop}[thm]{Proposition}
\newtheorem{remark}[thm]{Remark}

\numberwithin{thm}{section}

\newcommand{\bigzerou}{\smash{\lower1.7ex\hbox{\bg 0}}}

\newcommand{\bigastu}{\smash{\lower1.7ex\hbox{\bg *}}}

\newcounter{lastenumi}

\makeatletter
\newcommand{\mysloppy}{\tolerance 9999 \hfuzz .5\p@ \vfuzz .5\p@}
\makeatother
\title{%
Typical ranks of certain 3-tensors and \afcr\ tensors%
\footnote{%
The authors were supported partially by Grant-in-Aid for Scientific Research (B)
(No. 20340021) of the Japan Society for the Promotion of Science.}%
}
\author{Mitsuhiro Miyazaki\footnote{Kyoto University of Education, \tt g53448@kyokyo-u.ac.jp}
 \and Toshio Sumi\footnote{Kyushu University, \tt sumi@artsci.kyushu-u.ac.jp}
 \and Toshio Sakata\footnote{Kyushu University, \tt sakata@design.kyushu-u.ac.jp}%
}

\date{Version of \today}
\date{}

\begin{document}
\mysloppy

\maketitle

\begin{abstract}
In this paper, we study typical ranks of 3-tensors
and show that there are plural typical ranks for
$m\times n\times p$ tensors over $\RRR$ in the
following cases:
(1)
$3\leq m\leq \rho(n)$ and $(m-1)(n-1)+1\leq p\leq (m-1)n$,
where $\rho$ is the \hr\ function,
(2)
$m=3$, $n\equiv 3\pmod 4$ and $p=2n-1$,
(3)
$m=4$, $n\equiv 2\pmod 4$, $n\geq 6$ and $p=3n-2$,
(4)
$m=6$, $n\equiv 4\pmod 8$, $n\geq 12$ and $p=5n-4$.
(5)
$m=10$, $n\equiv 24\pmod{32}$ and $p=9n-8$.
\end{abstract}

\section{Introduction}

A tensor is another name for a high-dimensional array of datum.
Recently we have witnessed many applications of tensor data in broad 
fields
such as brain wave analysis, image analysis, web analysis and more.

Given a $k$-dimensional tensor $T=(t_{i_1 i_2 \cdots i_k})$ of size
$n_1\times \cdots\times n_k$ with entries in a field $\KKK$, we 
identify it with the
element $x\in \KKK^{n_1}\otimes\cdots\otimes\KKK^{n_k}$ such that
$x=\sum_{i_1=1}^{n_1}\cdots\sum_{i_k=1}^{n_k}t_{i_1\cdots i_k}\eee_{i_1}
\otimes\cdots\otimes\eee_{i_k}$,
where $\eee_i$ is the $i$-th fundamental vector.
Therefore
$x$ can be expressed as a sum of finite tensors of
form $\aaa_1\otimes\cdots\otimes\aaa_k$.
The rank of $x$ is the smallest number of the tensors of the form
$\aaa_1\otimes\cdots\otimes\aaa_k$ needed to express $x$ as their sum
of them.
In terms of high-dimensional array data,
 $T=(t_{i_1\cdots i_k})$
is identified with
$\aaa_1\otimes \cdots\otimes\aaa_k$
if and only if 
$t_{i_1\cdots i_k}=\prod_{j=1}^k a_{i_j}^{(j)}$,
where $\aaa_j=(a_1^{(j)},\ldots, a_{n_j}^{(j)})\transpose$ for $j=1$, \ldots, $k$.
Therefore, the rank of a 
tensor is a measure of its complexity.

So it is worth studying the maximal rank of tensors of a given size.
It is also important to know the ranks which appear with 
positive probability when the entries of a tensor with fixed
size vary randomly.
These ranks are called the typical ranks.
See for example \cite{Berge 2000}, 
\cite{Berge 2004}, \cite{Berge 2011}
and \cite{cbdc}.

In this paper, we consider typical ranks of 3-tensors, i.e.,
3-dimensional arrays of data. 
In particular we show the following fact
(see Theorems \ref{thm:ans case} and \ref{thm:misc cases}).
\begin{thm}
\mylabel{thm:intro}
There are at least two typical ranks of $m\times n\times p$ tensors
over $\RRR$ in the following cases.
\begin{enumerate}
\item
\mylabel{item:rho n}
$3\leq m\leq \rho(n)$ and $(m-1)(n-1)+1\leq p\leq (m-1)n$,
where $\rho$ is the \hr\ function.
\item
$m=3$, $n\equiv 3\pmod 4$ and $p=2n-1$.
\item
$m=4$, $n\equiv 2\pmod 4$, $n\geq 6$ and $p=3n-2$.
\item
$m=6$, $n\equiv 4\pmod 8$, $n\geq 12$ and $p=5n-4$.
\item
$m=10$, $n\equiv 24\pmod {32}$ and $p=9n-8$.
\end{enumerate}
\end{thm}
The case where $p=(m-1)n$ of \ref{item:rho n}
are already proved in \cite{ssm}.
Note that if $m\leq n$ and $p\geq (m-1)n+1$, then $\min\{p,mn\}$ is the
unique typical rank \cite{Berge 2000}.
Note also that $\min\{p, mn\}$ is the minimal typical rank
if and only if $p \geq (m-1)(n-1)+1$ \cite{cgg}.
In particular, 
in any case of Theorem \ref{thm:intro},
$p$ is the minimal typical rank.

In order to prove this theorem,
we introduce the concept of \afcr\ tensors.
It is a generalization of absolutely nonsingular tensors defined in \cite{ssm}.

\section{Preliminaries}

We first recall some basic facts and establish terminology.

\begin{notation}
\begin{enumerate}
\item We denote by $\KKK$ an arbitrary field and by $\FFF$ 
the real number field $\RRR$ or the complex number field $\CCC$.

\item We denote by $E_n$ the $n\times n$ identity matrix.

\item
For a tensor $x\in\KKK^m\otimes\KKK^n\otimes\KKK^p$ with
$x=\sum_{ijk}a_{ijk}\eee_i\otimes\eee_j\otimes\eee_k$,
we identify $x$ with
$T=(a_{ijk})_{1\leq i\leq m,1\leq j\leq n,1\leq k\leq p}$
and denote it
$(A_1;\cdots;A_p)$,
where $A_k=(a_{ijk})_{1\leq i\leq m,1\leq j\leq n}$ for $k=1$, \ldots, $p$ 
is an $m\times n$ matrix,
and call
$(A_1;\cdots;A_p)$ a tensor.

\item
We denote the set of $m\times n\times p$ tensors by
$\KKK^{m\times n\times p}$.

\item
For an $m\times n\times p$ tensor $T=(A_1;\cdots;A_p)$,
an 
$l\times m$ matrix $P$
and
an 
$n\times k$ matrix $Q$,
we denote by $PTQ$ the $l\times k\times p$ tensor
$(PA_1Q;\cdots;PA_pQ)$. 

\item
For $m\times n$ matrices $A_1$, \ldots, $A_p$,
we denote by $(A_1,  \ldots, A_p)$ the $m\times np$ matrix obtained by aligning
$A_1$, \ldots, $A_p$ horizontally.

\item
We set $\diag(A_1,A_2,\ldots,A_t)=
\begin{pmatrix} A_1&&&O \\ 
  & A_2\cr &&\ddots \\
  O &&& A_t \end{pmatrix}
$ 
for matrices $A_1$, $A_2$, \ldots, $A_t$ and
define $\diag(T_1,T_2,\ldots,T_t)$ similarly for tensors $T_1$, $T_2$, \ldots, $T_t$
with the same number of slices.

\item
For an $m\times n$ matrix $M$,
we denote the $m\times j$ 
(resp.\ $m\times (n-j)$)
matrix consisting of the first $j$ 
(resp.\ last $n-j$)
columns of $M$
by $M_{\leq j}$
(resp.\ ${}_{j<}M$).
We denote the $i\times n$ 
(resp.\ $(m-i)\times n$)
matrix consisting of the first $i$ 
(resp.\ last $m-i$)
rows of $M$
by $M^{\leq i}$
(resp.\ ${}^{i<}M$).
\item
For an $m\times n\times p$ tensor $T=(A_1;\cdots;A_p)$,
we set
$T_{\leq j}\define((A_1)_{\leq j};\cdots;(A_p)_{\leq j})$.
\item
For an $m \times n$ matrix $A=(a_{ij})$ and an $s\times t$ matrix $B$,
we denote the $ms\times nt$ matrix 
$$
\begin{pmatrix}
a_{11}B& \cdots & a_{1n}B\\
\vdots&&\vdots\\
a_{m1}B&\cdots&a_{mn}B
\end{pmatrix}
$$
by $A\otimes B$.
\item
Let $V$ and $W$ be algebraic varieties.
For a rational map $\varphi\colon V\hbox{$-$\,$-$\,$\to$}W$,
we denote the domain of $\varphi$ by $\dom(\varphi)$.
\end{enumerate}
\end{notation}

\begin{definition}\rm
Let $x$ be an element of $\KKK^m\otimes\KKK^n\otimes\KKK^p$.
We define the rank of $x$, denoted by $\rank\,x$, to be
$\min\{r\mid\exists \aaa_i\in\KKK^m$, $\exists\bbb_i\in\KKK^n$, $\exists\ccc_i\in\KKK^p$ for 
$i=1$, \ldots, $r$ such that 
$x=\sum_{i=1}^r\aaa_i\otimes\bbb_i\otimes\ccc_i\}$.
\end{definition}
If $\KKK'$ is an extension field of $\KKK$ and 
$x\in \KKK^m\otimes \KKK^n\otimes \KKK^p$, 
then we may regard $x$ 
as an element of ${\KKK'}^m\otimes{\KKK'}^n\otimes {\KKK'}^p$.
In order to distinguish the rank of $x$ as a tensor over $\KKK$ and the rank of $x$ as a tensor over $\KKK'$, we denote by $\rank_{\KKK}x$ and $\rank_{\KKK'}x$
respectively if necessary.

\begin{example}[{\cite[Example 2.9]{sms}}] 
\rm
\mylabel{eg:field}
For a $2\times 2$ matrix $A=\begin{pmatrix} 0&1\cr -1&0\end{pmatrix}$,
it holds that $\rank_{\RRR}(E_2;A)=3$ and $\rank_{\CCC}(E_2;A)=2$.
\end{example}

\begin{definition}\rm
Let $T=(A_1;\cdots;A_p)$ be a tensor.
We define 
the column rank, denoted $\crank T$, and the row rank,
denoted $\rrank T$, of $T$ by
$\crank T\define\rank\begin{pmatrix}A_1\\\vdots\\A_p\end{pmatrix}$
and $\rrank T\define\rank(A_1,\ldots, A_p)$ respectively.
\end{definition}

\begin{remark}\rm
\mylabel{rem:row column rank}
Let $T$ be a tensor.
Then $\rank T\geq \max\{\crank T$, $\rrank T\}$.
\end{remark}

\begin{definition}\rm
Two tensors $T$ and $T'$ are said to be {\em equivalent} if
there are nonsingular matrices $P$ and $Q$ such that $T'=PTQ$.
\end{definition}

\begin{remark}\rm
If $T$ and $T'$ are equivalent, then
$\rank T=\rank T'$.
\end{remark}

\begin{definition}\rm
Let $m$, $n$ and $p$ be positive integers.
If a generic $m\times n\times p$ tensor over $\FFF$ has rank $r$,
that is,
there is a Zariski dense open subset $U$ of $\FFF^{m\times n\times p}$ such that
$\rank T=r$ for any $T\in U$,
we say that the generic rank of  $m\times n\times p$ tensors over $\FFF$ is
$r$ and denote $\grank_\FFF(m,n,p)=r$.
\end{definition}
\begin{definition}\rm
We set
$\trank_\FFF(m,n,p)=\{r\mid$ there is a subset $S\subset\FFF^{m\times n\times p}$ such that
$S$ has positive Lebesgue measure and $\rank_\FFF T=r$ for any $T\in S\}$
and we call an element of $\trank_\FFF(m,n,p)$ a typical rank of $m\times n\times p$ tensors
over $\FFF$.
\end{definition}
We set $f_1\colon \FFF^m\times \FFF^n\times \FFF^p\to\FFF^{m\times n\times p}$
by
$$((x_1,\ldots, x_m)\transpose, (y_1,\ldots,y_n)\transpose, (z_1,\ldots, z_p)\transpose)
\mapsto (x_iy_jz_k)$$
and for $t>1$, we set
$f_t\colon (\FFF^m\times \FFF^n\times \FFF^p)^t\to \FFF^{m\times n\times p}$
by
$$((\xxx_1,\yyy_1,\zzz_1),\ldots,(\xxx_t,\yyy_t,\zzz_t))\mapsto
\sum_{u=1}^t f_1(\xxx_u,\yyy_u,\zzz_u).$$
Then for $T\in \FFF^{m\times n\times p}$,
$\rank T=\min\{t\mid T\in \image f_t\}$.
\begin{remark}\rm
\mylabel{rem:chevalley}
Consider the case where $\FFF=\CCC$.
Then by the theorem of Chevalley 
\cite{che} (see also \cite[Theorem 3.16]{har} or \cite[(2.31) Proposition]{mum}), 
$\image f_t$ is a 
constructible set of $\CCC^{m\times n\times p}$.
Therefore, the following conditions are equivalent.
\begin{enumerate}
\item
The Zariski closure of $\image f_t$ is $\CCC^{m\times n\times p}$.
\item
$\image f_t$ contains a Zariski dense open subset of $\CCC^{m\times n\times p}$.
\item
The Euclidean closure of $\image f_t$ is $\CCC^{m\times n\times p}$.
\end{enumerate}
In particular,
$\min\{t\mid$ the Euclidean closure of $\image f_t$ is $\CCC^{m\times n\times p}\}$
is the generic rank of $m\times n\times p$ tensors over $\CCC$.
\end{remark}

Here we recall the following result of Friedland.
\begin{thm}[{\cite[Theorem~7.1]{Friedland:2008}}] \label{thm:Friedland}
The space $\RRR^{m_1\times m_2\times m_3}$, $m_1,m_2,m_3 \in\mathbb{N}$, 
contains a finite number of open connected disjoint semi-algebraic sets 
$O_1,\ldots,O_M$ satisfying the following properties.
\begin{enumerate}
\item $\RRR^{m_1\times m_2\times m_3}\smallsetminus \cup_{i=1}^M O_i$
is a closed semi-algebraic set $\RRR^{m_1\times m_2\times m_3}$ 
of dimension less than $m_1m_2m_3$.
\item  Each $T \in O_i$ has rank $r_i$ for $i = 1,\ldots,M$.
\item The number $\min(r_1,\ldots,r_M)$ is equal to $\grank_\CCC(m_1,m_2,m_3)$.
\item $\max(r_1,\ldots,r_M)$ is the minimal 
$t\in \mathbb{N}$ such that the Euclidean 
closure of $f_t((\RRR^{m_1}\times \RRR^{m_2}\times \RRR^{m_3})^t)$ 
is equal to $\RRR^{m_1\times m_2\times m_3}$.
\item For each integer $r\in [\grank_\CCC(m_1,m_2,m_3),\max(r_1, \ldots, r_M)]$,
there exists $r_i=r$ for some integer $i\in [1,M]$.
\end{enumerate}
\end{thm}
Therefore, we see the following fact.
\begin{prop} 
\mylabel{prop:open}
Let $r$ be a positive integer.
Then $r\in\trank_\RRR(m_1$, $m_2$, $m_3)$ if and only if there is a 
non-empty Euclidean open subset
$U$ of $\RRR^{m_1\times m_2\times m_3}$ such that
for any $T\in U$, $\rank T=r$.
\end{prop}
\begin{proof}
``If'' part is immediate from the definition of typical rank.
Assume that $r$ is a typical rank.
Then there is 
a subset $S\subset\RRR^{m_1\times m_2\times m_3}$ such that
$S$ has positive Lebesgue measure and $\rank T=r$ for any $T\in S$.
Since
$\dim(\RRR^{m_1\times m_2\times m_3}\smallsetminus \cup_{i=1}^M O_i)<m_1m_2m_3$,
there is $i$ such that $S\cap O_i\neq\emptyset$.
Therefore, $O_i$ is a non-empty Euclidean open set 
such that $\rank T=r$ for any $T\in O_i$.
\end{proof}
In particular, we see the following:
\begin{remark}\rm
If there is the generic rank of $m\times n\times p$ tensors over $\FFF$, then it is the
unique typical rank of $m\times n\times p$ tensors over $\FFF$.
\end{remark}

\begin{remark}\rm
\mylabel{rem:grankc}
It is known that if
$p\geq (m-1)(n-1)+1$, then
$\grank_\CCC(m,n,p)=\min\{p,mn\}$
(cf. 
\cite[Theorem 2.4 and Remark 2.5]{cgg}).
\end{remark}

\section{\Afcr\ tensors}

First we recall the following definition.

\begin{definition}[{\cite{ssm}}]
\rm
Let
$T=(A_1;\cdots;A_p)$
be an $n\times n\times p$ tensor over $\RRR$.
$T$ is called an absolutely nonsingular tensor if the equation
$$
\det(\sum_{k=1}^p x_k A_k)=0
$$
implies $x_1=x_2=\cdots=x_p=0$.
\end{definition}
We generalize this notion and state the following:

\begin{definition}
\rm
\mylabel{def:afr}
Let
$T=(A_1;\cdots;A_p)$
be an $l\times n\times p$ tensor over $\RRR$.
$T$ is called an \afcr\ tensor or simply an \afr\ tensor if
$$
\rank(\sum_{k=1}^p x_k A_k)=n
$$
for any $(x_1,x_2,\ldots,x_p)\in\RRR^p\setminus\{(0,0,\ldots,0)\}$.
\end{definition}
It follows from the definition
that a tensor which is equivalent to an \afr\ tensor is also \afr.
Next we note the following lemma whose proof is 
straightforward.
\begin{lemma}
\mylabel{lem:char afr}
Let
$T=(A_1;\cdots;A_p)$
be an $l\times n\times p$ tensor over $\RRR$.
Set
$A_i=(\aaa_{i1},\ldots,\aaa_{in})$
for $i=1$, \ldots, $p$.
Then the following conditions are equivalent.
\begin{enumerate}
\item
$T$ is \afr.
\item
If
$
\sum_{i=1}^p\sum_{j=1}^n x_iy_j\aaa_{ij}=\zerovec
$,
where $x_1$, \ldots, $x_p$, $y_1$, \ldots, $y_n\in\RRR$,
then $x_1=\cdots=x_p=0$ or $y_1=\cdots=y_n=0$.
\item
$
\sum_{i=1}^p x_iA_i\yyy\neq \zerovec
$
for any $\xxx=(x_1,\ldots, x_p)\transpose\in S^{p-1}$ and
$\yyy\in S^{n-1}$,
where $S^d$ stands for the $d$-dimensional sphere.
\end{enumerate}
\end{lemma}
As a corollary, we see that a tensor obtained by rotating
an \afr\  tensor by $90^\circ$ is also \afr.
To be precise, we see the following fact.
\begin{cor}
\mylabel{cor:afr roll}
Let
$T=(A_1;\cdots;A_p)$
be an $l\times n\times p$ tensor over $\RRR$.
Set
$A_i=(\aaa_{i1},\ldots,\aaa_{in})$
for $i=1$, \ldots, $p$.
And set
$B_j=(\aaa_{pj},\aaa_{p-1,j},\ldots,\aaa_{1j})$
for $j=1$, \ldots, $n$
and 
$T'=(B_1;\cdots;B_n)$.
Then $T$ is \afr\ if and only if
so is $T'$.
In particular, there is an $l\times n\times p$ \afr\ tensor
if and only if there is an $l\times p\times n$ \afr\ tensor.
\end{cor}
%
%

Now we prove the following important fact.
\begin{thm}\mylabel{thm:afr open}
Let $l$, $n$, and $p$ be positive integers.
Then the set 
$\{T\in\RRR^{l\times n\times p}\mid
T$ is \afr $\}$
is a (possibly empty) open subset of
$\RRR^{l\times n\times p}$
in the Euclidean topology.
\end{thm}
\begin{proof}
Let $T$ be an $l\times n\times p$ \afr\ tensor.
If $T$ is not an interior point of the
set in question, there is a sequence
$\{T_k\}$ of
tensors of size $l\times n\times p$
such that
$
T_k\to T
$ and
$T_k$ is not \afr\ for any $k$.

Set $T=(A_1;\cdots;A_p)$ and $T_k=(A_1^{(k)};\cdots;A_p^{(k)})$ for each $k$.
Since $T_k$ is not \afr,
we see by Lemma \ref{lem:char afr} that there are
$\xxx^{(k)}=(x_1^{(k)},\ldots,x_p^{(k)})\transpose\in S^{p-1}$ and
$\yyy^{(k)}\in S^{n-1}$ such that
$$
\sum_{i=1}^p x_i^{(k)}A_i^{(k)}\yyy^{(k)}
= \zerovec
$$
for any $k$.
Since $S^{p-1}$ and $S^{n-1}$ are compact, we may assume,
by taking subsequences if necessary, that
$\{\xxx^{(k)}\}$ and $\{\yyy^{(k)}\}$ converge.

Set
$\xxx=\lim_{k\to\infty}\xxx^{(k)}$, $\yyy=\lim_{k\to\infty}\yyy^{(k)}$
and $\xxx=(x_1,\ldots,x_p)\transpose$.
Then $\xxx\in S^{p-1}$, $\yyy\in S^{n-1}$ and
$$
\sum_{i=1}^p x_iA_i\yyy= 
\lim_{k\to\infty}\sum_{i=1}^p x_i^{(k)}A_i^{(k)}\yyy^{(k)}= \zerovec.
$$
This contradicts the fact that $T$ is \afr.
\end{proof}

\section{Typical ranks of certain 3-tensors}

In this section, we consider typical ranks of 3-tensors with 
fixed sizes with a certain condition.
First consider the following condition of a sequence of matrices:

\begin{definition}
\rm
\mylabel{cond:afr parts m}
Let $n$, $l$, and $m$ be integers with $0\leq l< n$ and $m\geq 3$.
Also let
$A$ be an $n\times (2n-l)$ matrix and
$A_3$, $A_4$, \ldots, $A_m$ $n\times n$ matrices
with entries in $\RRR$.
Set $A=(B_1$, $B_0$, $B_2)$,
where $B_1$ and $B_2$ are $n\times (n-l)$ matrices and $B_0$ is
an $n\times l$ matrix.
If
$\left(
\begin{pmatrix}B_1&O\\ B_1&B_0\end{pmatrix};
\begin{pmatrix}B_0&B_2\\ O&B_2\end{pmatrix};
\begin{pmatrix}A_3\\ A_3\end{pmatrix};\cdots;
\begin{pmatrix}A_m\\ A_m\end{pmatrix}
\right)$
is \afr, then we say that the sequence of matrices $A$, $A_3$, \ldots, 
$A_m$ satisfies Condition 
\thethm\ with respect to $n$, $l$, $m$.
\end{definition}
Using this notion, we state the following:

\begin{thm}
\mylabel{thm:main m}
Let $m$, $n$, and $p$ be positive integers with 
$m\geq 3$ and $(m-2)n<p\leq (m-1)n$.
Set $l=(m-1)n-p$.
If there is  a sequence of matrices
$A$, $A_3$, \ldots, $A_{m}$
satisfying Condition \ref{cond:afr parts m}
with respect to $n$, $l$, $m$,
then  $\trank_\RRR(n,p,m)$ contains a number 
larger than $p$.
\end{thm}

In order to prove Theorem \ref{thm:main m}, 
we first recall our previous results.
%
%
%
\begin{thm}[{\cite[Theorem 8]{mss}}]
\label{thm:NOLTA 8}
Let $\KKK$ be an infinite field, and
$s$ and $t$ integers with $0<s<t$.
Then there are rational maps $\varphi^{(1)}$ and $\varphi^{(2)}$ 
from $\KKK^{s\times t\times 2}$ to
$\glin(s;\KKK)$ and $\glin(t;\KKK)$ respectively,
such that
\begin{eqnarray*}
&&((O,E_s);(E_s,O))\in\dom(\varphi^{(1)})\cap\dom(\varphi^{(2)}),\\
&&\varphi^{(1)}((O,E_s);(E_s;O))=E_s,\\
&&\varphi^{(2)}((O,E_s);(E_s;O))=E_t, 
\end{eqnarray*}
and 
$$
\varphi^{(1)}(T)T\varphi^{(2)}(T)=((O,E_s);(E_s;O))
$$ 
for any $T\in\dom(\varphi^{(1)})\cap\dom(\varphi^{(2)})$.
\end{thm}
By considering $\varphi^{(1)}(A_2;A_1)$ and $\varphi^{(2)}(A_2;A_1)$,
where $T=(A_1;A_2)$,
we see the following:
\begin{thm}
\mylabel{thm:rational map}
Let $\KKK$ be an infinite field, and
$s$ and $t$ integers with $0<s<t$.
Then there are rational maps $\varphi^{(1)}$ and $\varphi^{(2)}$ 
from $\KKK^{s\times t\times 2}$ to
$\glin(s;\KKK)$ and $\glin(t;\KKK)$ respectively,
such that
\begin{eqnarray*}
&&((E_s,O);(O,E_s))\in\dom(\varphi^{(1)})\cap\dom(\varphi^{(2)}),\\
&&\varphi^{(1)}((E_s,O);(O,E_s))=E_s,\\
&&\varphi^{(2)}((E_s,O);(O,E_s))=E_t,
\end{eqnarray*}
and
$$
\varphi^{(1)}(T)T\varphi^{(2)}(T)=((E_s,O);(O,E_s))
$$
for any $T\in\dom(\varphi^{(1)})\cap\dom(\varphi^{(2)})$.
\end{thm}
Note that through this theorem, we see that a generic $s\times t\times 2$
tensor is equivalent to $((E_s,O);(O,E_s))$.
Therefore, this theorem gives another proof of the result
of \cite{Berge and Kiers 1999}.
We also recall the following:
\begin{lemma}[{\cite[Lemma 9]{mss}}]
\label{lem:NOLTA 9}
Let $\KKK$ be an infinite field and
$s$, $t$ and $u$ integers with $0<s<t$.
Then there is a rational map $\varphi_0$ from $\KKK^{s\times t\times u}$
to $\glin(t;\KKK)$ such that
$$
((O,E_s);A_2;\cdots;A_u)\in\dom(\varphi_0)\mbox{ and }
\varphi_0((O,E_s);A_2;\cdots;A_u)=E_t
$$
for any $s\times t$ matrices $A_2$, \ldots, $A_u$
and
$$
(A_1;\cdots;A_u)\varphi_0((A_1;\cdots;A_u))=((O,E_s);\ast;\cdots;\ast)
$$
for any $(A_1;\cdots;A_u)\in\dom(\varphi_0)$.
\end{lemma}
By considering $\varphi_0(A_u;A_1;\cdots;A_{u-1})$, we see the following:
\begin{lemma}
\mylabel{lem:prep}
Let $\KKK$ be an infinite field and
$s$, $t$ and $u$ integers with $0<s<t$.
Then there is a rational map $\varphi_0$ from $\KKK^{s\times t\times u}$
to $\glin(t;\KKK)$ such that
$$
(A_1;\cdots;A_{u-1};(O,E_s))\in\dom(\varphi_0)
\mbox{ and }
\varphi_0(A_1;\cdots;A_{u-1};(O,E_s))=E_t
$$
for any $s\times t$ matrices $A_1$, $A_2$, \ldots, $A_{u-1}$
and
$$
(A_1;\cdots;A_{u})\varphi_0((A_1;\cdots;A_u))=(\ast;\cdots;\ast;(O,E_s))
$$
for any $(A_1;\cdots;A_u)\in\dom(\varphi_0)$.
\end{lemma}
Now we state the following result
which is easily proved by
Lemma \ref{lem:prep}, Theorem \ref{thm:rational map} and column operations:
\begin{thm}
\mylabel{thm:rational map gen}
Let $\KKK$ be an infinite field, and
$s$, $t$ and $u$ positive integers with $u\geq 2$ and $(u-1)s<t$.
Set $v=t-(u-1)s$ and
$X=(X_1;\cdots;X_u)$, where
$X_1=(E_s,O_{s\times (t-s)})$,
$X_2=(O_{s\times v}, E_s, O_{s\times (u-2)s})$,
$X_3=(O_{s\times (v+s)}, E_s, O_{s\times (u-3)s})$,
$X_4=(O_{s\times (v+2s)}, E_s, O_{s\times (u-4)s})$,
\ldots,
$X_u=(O_{s\times (t-s)}, E_s)$.
Then there are rational maps $\psi^{(1)}$ and $\psi^{(2)}$ from 
$\KKK^{s\times t\times u}$ to
$\glin(s;\KKK)$ and $\glin(t;\KKK)$ respectively,
such that
$X\in\dom(\psi^{(1)})\cap\dom(\psi^{(2)})$,
$\psi^{(1)}(X)=E_s$,
$\psi^{(2)}(X)=E_t$
and for any $T\in\dom(\psi^{(1)})\cap\dom(\psi^{(2)})$
$$
\psi^{(1)}(T)T\psi^{(2)}(T)=((E_s,O_{s\times v},M);X_2;X_3;\cdots;X_u),
$$
where $M$ is an $s\times (u-2)s$ matrix with
$M^{\leq v}=O$ if $t<us$ or $M=O$ if $t\geq us$.
\end{thm}

Now we state the proof of Theorem \ref{thm:main m}:

\begin{proofof}{Theorem \ref{thm:main m}}
We may assume that $A_m=E_n$.
Set 
$Y_1=(E_n,O_{n\times (p-n)})$,
$Y_2=(O_{n\times (n-l)},E_n,O_{n\times (m-3)n})$,
$Y_3=(O_{n\times (2n-l)},E_n,O_{n\times (m-4)n})$,
$Y_4=(O_{n\times (3n-l)},E_n,O_{n\times (m-5)n})$,
\ldots,
$Y_{m-1}=(O_{n\times (p-n)},E_n)$,
$Y_m=(A,A_3,A_4,\ldots,A_{m-1})$
and
$Y=(Y_1;Y_2;\cdots;Y_m)$.
Let $\psi^{(1)}$ and $\psi^{(2)}$ be rational maps from
$\RRR^{n\times p\times (m-1)}$ to $\glin(n,\RRR)$ and $\glin(p,\RRR)$
respectively of Theorem \ref{thm:rational map gen}.

Consider the set $U$ of $n\times p\times m$ tensors
$(X_1;X_2;\cdots;X_m)$ over $\RRR$ such that 
$X=(X_1;\cdots;X_{m-1})\in\dom(\psi^{(1)})\cap\dom(\psi^{(2)})$ and if we set
$V=\begin{pmatrix}\psi^{(1)}(X)X_1\psi^{(2)}(X)\\ O\ \ E_{p-n}\end{pmatrix}^{-1}$,
then
$$
\left(
\begin{pmatrix}B_1&O\\C_1&C_0\end{pmatrix};
\begin{pmatrix}B_0&B_2\\O&C_2\end{pmatrix};
\begin{pmatrix}B_3\\C_3\end{pmatrix};
\cdots;
\begin{pmatrix}B_{m-1}\\C_{m-1}\end{pmatrix};
\begin{pmatrix}E_n\\E_n\end{pmatrix}
\right)
$$
is \afr,
where 
$\psi^{(1)}(X)X_m\psi^{(2)}(X)=(B_1,B_0,B_2,B_3,\ldots,B_{m-1})$,
$\psi^{(1)}(X)X_m\psi^{(2)}(X)V=(C_1,C_0,C_2,C_3,\ldots,C_{m-1})$,
$B_0$ and $C_0$ are $n\times l$ matrices,
$B_1$, $B_2$, $C_1$ and $C_2$ are $n\times (n-l)$ matrices
and 
$B_3$, \ldots, $B_{m-1}$ and $C_3$, \ldots, $C_{m-1}$ are
$n\times n$ matrices.
Then we see that $U$ is  a Euclidean open set containing $Y$ by 
Theorem \ref{thm:afr open}
since rational maps are continuous.

Now we
\begin{claim}
If $T\in U$, then $\rank_\RRR T>p$.
\end{claim}
Assume the contrary and take $T\in U$ with $\rank T\leq p$.
Set
$T=(X_1;X_2;\cdots;X_m)$,
$X=(X_1;X_2;\cdots;X_{m-1})$,
$V$, $B_0$, $B_1$, \ldots, $B_{m-1}$, $C_0$, $C_1$, \ldots, $C_{m-1}$
as above and
$\psi^{(1)}(X)T\psi^{(2)}(X)=Z=(Z_1;Z_2;\cdots;Z_m)$.
Then by the definition of $U$,
$$
\left(
\begin{pmatrix}B_1&O\\C_1&C_0\end{pmatrix};
\begin{pmatrix}B_0&B_2\\O&C_2\end{pmatrix};
\begin{pmatrix}B_3\\C_3\end{pmatrix};
\cdots;
\begin{pmatrix}B_{m-1}\\C_{m-1}\end{pmatrix};
\begin{pmatrix}E_n\\E_n\end{pmatrix}
\right)
$$
is \afr.

Since $Z$ and $T$ are equivalent, 
$\rank Z=\rank T\leq p$.
On the other hand,
since $\crank Z=p$, we see that $\rank Z\geq p$.
So there are an $n\times p$ matrix $P$, a $p\times p$ matrix $Q$ and
$p\times p$ diagonal matrices $D_k$ with
$Z_k=PD_kQ$ for $k=1,2,\ldots, m$.
Since
$$E_p=\begin{pmatrix}Z_1^{\leq n-l}\\Z_2\\\vdots\\Z_{m-1}\end{pmatrix}
=\begin{pmatrix}P^{\leq n-l}D_1\\PD_2\\\vdots\\PD_{m-1}\end{pmatrix}Q,
$$
we see that $Q$ is nonsingular and
$$Q^{-1}=
\begin{pmatrix}P^{\leq n-l}D_1\\PD_2\\\vdots\\PD_{m-1}\end{pmatrix}.
$$
Moreover, since
$$V^{-1}=
\begin{pmatrix}Z_1\\{}^{l<}Z_2\\Z_3\\\vdots\\Z_{m-1}\end{pmatrix}=
\begin{pmatrix}PD_1\\{}^{l<}PD_2\\PD_3\\\vdots\\PD_{m-1}\end{pmatrix}Q,
$$
we see that
$$V^{-1}Q^{-1}=
\begin{pmatrix}PD_1\\{}^{l<}PD_2\\PD_3\\\vdots\\PD_{m-1}\end{pmatrix}.
$$

Now set
$P=(\uuu_1,\uuu_2,\ldots,\uuu_p)$ and
$D_k=\diag(d_{k1},d_{k2},\ldots,d_{kp})$ for $k=1,2,\ldots, m$.
Then
\begin{eqnarray*}
PD_m&=&Z_mQ^{-1}\\
&=&B_1P^{\leq n-l}D_1+(B_0,B_2)PD_2+B_3PD_3+\cdots+B_{m-1}PD_{m-1}\\
&=&(B_1,O)PD_1+(B_0,B_2)PD_2+B_3PD_3+\cdots+B_{m-1}PD_{m-1}.
\end{eqnarray*}
On the other hand, we see
\begin{eqnarray*}
PD_m&=&(Z_mV)(V^{-1}Q^{-1})\\
&=&(C_1,C_0)PD_1+C_2{}^{l<}PD_2+C_3PD_3+\cdots+C_{m-1}PD_{m-1}\\
&=&(C_1,C_0)PD_1+(O,C_2)PD_2+C_3PD_3+\cdots+C_{m-1}PD_{m-1}.
\end{eqnarray*}
In particular,
\begin{equation}
\mylabel{eqn:pdm}
\begin{array}{rcl}
\displaystyle
\begin{pmatrix}B_1&O\\ C_1&C_0\end{pmatrix}PD_1&+&
\displaystyle
\begin{pmatrix}B_0&B_2\\ O&C_2\end{pmatrix}PD_2+
\begin{pmatrix}B_3\\ C_3\end{pmatrix}PD_3+
\cdots\\
\cdots&+&
\displaystyle
\begin{pmatrix}B_{m-1}\\ C_{m-1}\end{pmatrix}PD_{m-1}-
\begin{pmatrix}E_n\\ E_n \end{pmatrix}PD_m=O.
\end{array}
\end{equation}

Since $\rank Z=p$,
we see that $\uuu_1\neq\zerovec$ and $(d_{11},d_{21},\ldots, d_{m1})\neq(0,0,\ldots, 0)$.
Therefore 
\begin{eqnarray*}
\rank\bigg(
d_{11}\begin{pmatrix}B_1&O\\ C_1&C_0\end{pmatrix}+
d_{21}\begin{pmatrix}B_0&B_2\\ O&C_2\end{pmatrix}+
d_{31}\begin{pmatrix}B_3\\ C_3\end{pmatrix}+\cdots&\\
+
d_{m-1,1}\begin{pmatrix}B_{m-1}\\ C_{m-1}\end{pmatrix}-
d_{m1}\begin{pmatrix}E_n\\ E_n \end{pmatrix}
\bigg)&=&n,
\end{eqnarray*}
since 
$\left(
\begin{pmatrix}B_1&O\\ C_1&C_0\end{pmatrix};
\begin{pmatrix}B_0&B_2\\ O&C_2\end{pmatrix};
\begin{pmatrix}B_3\\ C_3\end{pmatrix};
\cdots;
\begin{pmatrix}B_{m-1}\\ C_{m-1}\end{pmatrix};
\begin{pmatrix}E_n\\ E_n \end{pmatrix}
\right)$
is \afr.
However, by observing the first column of equation (\ref{eqn:pdm}),
we see that
\begin{eqnarray*}
\bigg(
d_{11}\begin{pmatrix}B_1&O\\ C_1&C_0\end{pmatrix}+
d_{21}\begin{pmatrix}B_0&B_2\\ O&C_2\end{pmatrix}+
d_{31}\begin{pmatrix}B_3\\ C_3\end{pmatrix}+
\cdots&\\
+
d_{m-1,1}\begin{pmatrix}B_{m-1}\\ C_{m-1}\end{pmatrix}-
d_{m1}\begin{pmatrix}E_n\\ E_n \end{pmatrix}
\bigg)
\uuu_1&=&\zerovec.
\end{eqnarray*}
This is a contradiction.
\end{proofof}

%
%

\section{Existence of sequences of matrices with Condition \ref{cond:afr parts m}
and plural typical ranks}

In this section, we argue for the existence of a sequence of matrices
with Condition \ref{cond:afr parts m}
and apply the result to show the existence of plural typical ranks in
some sizes of 3-tensors.

First we recall the condition of the sizes of which an
\ans\ tensor exists.

\begin{definition}\rm
Let $n$ be a positive integer.
Set $n=(2a+1)2^{b+4c}$, where $a$, $b$ and $c$ are integers with
$0\leq b<4$.
Then we define
$\rho(n)\define 8c+2^b$.
\end{definition}
$\rho(n)$ is called the \hr\ function.
Now we recall the following:

\begin{thm}[{\cite[Theorem 2.2]{ssm}}]
\mylabel{thm:ssm22}
There exists an $n\times n\times p$ absolutely nonsingular tensor if and only if $p\leq\rho(n)$.
\end{thm}

For later use we recall a method which can construct an $n\times n\times \rho(n)$
\ans\ tensor explicitly for the case where $n=2^d$ for some positive integer $d$.
First we state the following:

\begin{definition}\rm
Let $\{A_1, \ldots, A_s\}$ be a family of $n\times n$ matrices with entries in $\RRR$.
If
\begin{enumerate}
\item
$A_iA_i\transpose=E_n$ for $1\leq i\leq s$,
\item
$A_i=-A_i\transpose$ for $1\leq i\leq s$ and
\item
$A_iA_j=-A_jA_i$ for $i\neq j$,
\end{enumerate}
then we say that $\{A_1, \ldots, A_s\}$ is a \hr\ family of order $n$.
\end{definition}
The following result immediately follows from the definition.

\begin{lemma}
\mylabel{lem:subfam}
A subfamily of a \hr\ family is a \hr\ family.
\end{lemma}

Next we note the following lemma which is easily verified.

\begin{lemma}
\label{lem:hr to ans}
Let $\{A_1, \ldots, A_s\}$ be a \hr\ family of order $n$.
Set $A_{s+1}=E_n$.
Then for any $x_1$, \ldots, $x_{s+1}\in \RRR$,
$$
\left(\sum_{k=1}^{s+1}x_kA_k\right)
\left(\sum_{k=1}^{s+1}x_kA_k\right)\transpose
=(x_1^2+\cdots+x_s^2+x_{s+1}^2)E_n.
$$
In particular, $(A_1;\cdots;A_s;E_n)$ is an $n\times n\times (s+1)$ \ans\ tensor.
\end{lemma}

Set
$$
A=\begin{pmatrix}0&1\\-1&0\end{pmatrix},\quad
P=\begin{pmatrix}0&1\\1&0\end{pmatrix},\quad
Q=\begin{pmatrix}1&0\\0&-1\end{pmatrix}.
$$
Then the the following results hold.

\begin{prop}[{\cite[Proposition 1.5]{gs}}]
\label{prop:gs15}
\begin{enumerate}
\item
\mylabel{item:2}
$\{A\}$ is a \hr\ family of order $2$.
\item
\mylabel{item:4}
$\{A\otimes E_2, P\otimes A, Q\otimes A\}$ is a \hr\ family of order $4$.
\item
\mylabel{item:8}
$\{E_2\otimes A\otimes E_2, E_2\otimes P\otimes A, Q\otimes Q\otimes A,
P\otimes Q\otimes A, A\otimes P\otimes Q, A\otimes P\otimes P, A\otimes Q\otimes E_2\}$
is a \hr\ family of order $8$.
\end{enumerate}
\end{prop}

\begin{thm}[{\cite[Theorem 1.6]{gs}}]
\label{thm:gs16}
Let $\{M_1, \ldots, M_s\}$ be a \hr\ family of order $n$.
Then
\begin{enumerate}
\item
$\{A\otimes E_n, Q\otimes M_1, \ldots, Q\otimes M_s\}$ is a \hr\ family of order $2n$.
\item
If moreover, $\{L_1, \ldots, L_t\}$ is a \hr\ family of order $m$, then
$\{P\otimes M_1\otimes E_m, \ldots, P\otimes M_s\otimes E_m, 
Q\otimes E_n\otimes L_1, \ldots, Q\otimes E_n \otimes L_t, A\otimes E_{mn}\}$
is a \hr\ family of order $2nm$.
\end{enumerate}
\end{thm}

Now we state a criterion of the existence of a sequence of matrices with 
Condition \ref{cond:afr parts m}.

\begin{lemma}
\mylabel{lem:existence m}
Let $n$, $l$, and $m$ be integers with $0\leq l< n$ and $m\geq 3$.
Then the following conditions are equivalent.
\begin{enumerate}
\item
\mylabel{item:exist seq}
There is a sequence of matrices 
satisfying Condition \ref{cond:afr parts m} with respect to $n$, $l$, $m$.
\item
\mylabel{item:bottom O}
There are $(n+l)\times n$ matrices $C_1$, $C_2$
and $n\times n$ matrices $A_3$, \ldots, $A_m$ such that
$\left(
C_1;C_2;
\begin{pmatrix}A_3\\ O\end{pmatrix};\cdots
;\begin{pmatrix}A_m\\ O\end{pmatrix}
\right)$
is \afr.
\setcounter{lastenumi}{\value{enumi}}
\end{enumerate}
Moreover, if $m=3$, then the above conditions are equivalent
to the following one.
\begin{enumerate}
\setcounter{enumi}{\value{lastenumi}}
\item
\mylabel{item:exist afr}
There is an $(n+l)\times n\times 3$ \afr\ tensor.
\setcounter{lastenumi}{\value{enumi}}
\end{enumerate}
\end{lemma}
\begin{proof}
We will prove that \ref{item:exist seq} and \ref{item:bottom O} are equivalent to
the following conditions.
\begin{enumerate}
\setcounter{enumi}{\value{lastenumi}}
\item
\mylabel{item:afr parts m1}
There are $n\times (n-l)$ matrices $B_1$, $B_2$, 
an $n\times l$ matrix $B_0$ and $n\times n$ matrices $A_3$, \ldots, $A_m$
such that
$\left(
\begin{pmatrix}B_1&O\\ O&B_0\end{pmatrix};
\begin{pmatrix}B_0&B_2\\ -B_0&O\end{pmatrix};
\begin{pmatrix}A_3\\ O\end{pmatrix};\cdots
;\begin{pmatrix}A_m\\ O\end{pmatrix}
\right)$
is \afr.
\item
\mylabel{item:afr parts m2}
There are $n\times (n-l)$ matrices $B_1$, $B_2$,
an $n\times l$ matrix $B_0$ and $n\times n$ matrices $A_3$, \ldots, $A_m$
such that the columns of $B_0$ are linearly independent and
$\left(
\begin{pmatrix}B_1&O\\ O&E_l\end{pmatrix};
\begin{pmatrix}B_0&B_2\\ -E_l&O\end{pmatrix};
\begin{pmatrix}A_3\\ O\end{pmatrix};\cdots
;\begin{pmatrix}A_m\\ O\end{pmatrix}
\right)$
is \afr.
\item
\mylabel{item:0 e -e 0}
There are $n\times (n-l)$ matrices $B_1$, $B_2$,
$n\times l$ matrices $A_1$, $A_2$ and $n\times n$ matrices $A_3$, \ldots, $A_m$
such that 
$\left(
\begin{pmatrix}B_1&A_1\\ O&E_l\end{pmatrix};
\begin{pmatrix}A_2&B_2\\ -E_l&O\end{pmatrix};
\begin{pmatrix}A_3\\ O\end{pmatrix};\cdots
;\begin{pmatrix}A_m\\ O\end{pmatrix}
\right)$
is \afr.
\end{enumerate}

\ref{item:exist seq}$\Longleftrightarrow$\ref{item:afr parts m1}$\Longleftarrow$\ref{item:afr parts m2}$
\Longrightarrow$\ref{item:0 e -e 0}$\Longrightarrow$\ref{item:bottom O}
are easy.
Furthermore, in the case where $m=3$,
\ref{item:bottom O}$\Longleftrightarrow$\ref{item:exist afr} is also easily verified.
For \ref{item:afr parts m1}$\Longrightarrow$\ref{item:afr parts m2},
note that if the columns of $B_0$ are linearly dependent, then
$\left(
\begin{pmatrix}B_1&O\\ O&B_0\end{pmatrix};
\begin{pmatrix}B_0&B_2\\ -B_0&O\end{pmatrix};
\begin{pmatrix}A_3\\ O\end{pmatrix};\cdots
;\begin{pmatrix}A_m\\ O\end{pmatrix}
\right)$
is not \afr.
For \ref{item:0 e -e 0}$\Longrightarrow$\ref{item:afr parts m2},
one may assume by Theorem \ref{thm:afr open} that the columns of $A_1+A_2$
are linearly independent.
Then \ref{item:afr parts m2} is deduced from row operations.
For \ref{item:bottom O}$\Longrightarrow$\ref{item:0 e -e 0},
one may assume by Theorem \ref{thm:afr open} that 
$(-{}^{n<}C_2;{}^{n<}C_1)$ is in the intersection of the domains of
$\varphi^{(1)}$ and $\varphi^{(2)}$ of Theorem \ref{thm:rational map}.
Then \ref{item:0 e -e 0} follows by Theorem \ref{thm:rational map}.
\end{proof}
In view of this result, we state the following:

\begin{definition}\rm
\label{cond:sp afr}
Let $n$, $l$ and $m$ be integers with $0\leq l<n$ and $m\geq 3$
and 
$T=(C_1;C_2;\cdots;C_m)$ an $(n+l)\times n\times m$-tensor over $\RRR$.
If $T$ is \afr\ and  ${}^{n<}C_i=O$ for $3\leq i\leq m$, we say that 
$T$ satisfies Condition \ref{cond:sp afr}.
\end{definition}
By Theorem \ref{thm:main m}  and Lemma \ref{lem:existence m}, we see the following:

\begin{cor}
\mylabel{cor:5y}
Let $m$, $n$ and $p$ be integers with $m\geq 3$ and $(m-2)n<p\leq (m-1)n$.
Set $l=(m-1)n-p$.
If there is an $(n+l)\times n\times m$-tensor with Condition \ref{cond:sp afr},
then $\trank_\RRR(m,n,p)$ contains a number larger than $p$.
\end{cor}

\begin{remark}
\rm
\mylabel{rem:more than one}
Suppose $(m-1)(n-1)+1\leq p \leq (m-1)n$.
Then $\min(\trank_\RRR(n,p,m))=
\min(\trank_\RRR(m,n,p))=\grank_\CCC(m,n,p)=p$
by Theorem \ref{thm:Friedland} and Remark \ref{rem:grankc}.
Therefore if 
there is an 
integer larger than $p$ in $\trank_\RRR(m,n,p)$,
then there are at least two typical ranks.
\end{remark}

Here we state some basic facts which are immediately verified.
\begin{lemma}
\mylabel{lem:afr add cut}
Let $T=(A_1;A_2;\cdots;A_p)$ be 
an $l\times n\times p$ \afr\ tensor.
\begin{enumerate}
\item
\mylabel{item:add}
For any positive integer $k$,
$
\left(
\begin{pmatrix}A_1\\ O\end{pmatrix};
\begin{pmatrix}A_2\\ O\end{pmatrix};\cdots;
\begin{pmatrix}A_p\\ O\end{pmatrix}
\right)
$
is 
an $(l+k)\times n\times p$ \afr\ tensor,
where $O$ is a $k\times n$ zero matrix.
\item
\mylabel{item:cut}
For any integer $k$ with
$1\leq k\leq n-1$,
$T_{\leq k}$
is 
an $l\times k\times p$ \afr\ tensor.
\item
\mylabel{item:tensor}
For any integer $u$,
$(E_u\otimes A_1;\cdots;E_u\otimes A_p)$ is a
$ul\times un\times p$ \afr\ tensor.
\end{enumerate}
\end{lemma}
%
%
%
%
\begin{cor}
\mylabel{cor:suf exist seq}
Let $n$, $l$ and $m$ be integers with $0\leq l<n$ and $m\geq 3$.
\begin{enumerate}
\item
\mylabel{item:ans exist seq}
If there is an $n\times n\times m$ absolutely nonsingular tensor,
then there is an $(n+l)\times n\times m$ tensor 
with Condition \ref{cond:sp afr}.
\item
If there is an $(n+l)\times n\times m$ tensor 
with  Condition \ref{cond:sp afr}, then
there is an $(n+l')\times n\times m$ tensor 
with Condition \ref{cond:sp afr}
for any $l'$ with $l<l'<n$.
\end{enumerate}
\end{cor}
By Corollaries \ref{cor:5y} and \ref{cor:suf exist seq}, we see the following:
\begin{cor}
\mylabel{cor:ans case}
Let $m$ and $n$ be integers with $3\leq m\leq \rho(n)$.
Then $\trank_\RRR(m,n,p)$ contains a number larger than $p$
for any $p$ with $(m-2)n<p\leq (m-1)n$.
\end{cor}
Therefore, we see the following result by
Remark \ref{rem:more than one}.
\begin{thm}
\mylabel{thm:ans case}
Let $m$, $n$ and $p$ be integers with $3\leq m\leq \rho(n)$ and 
$(m-1)(n-1)+1\leq p\leq (m-1)n$.
Then 
there are at least two typical ranks of $m\times n\times p$ tensors
over $\RRR$.
\end{thm}
We also obtain the following:
\begin{thm}
\mylabel{thm:misc cases}
Let $m$, $n$ be integers with $m\leq n$.
Set $p=(m-1)(n-1)+1$.
Then there are at least two typical ranks of $m\times n\times p$
tensors over $\RRR$  in the following cases.
\begin{enumerate}
\item
\mylabel{item:m=3}
$m=3$, $n\equiv 3\pmod 4$.
\item
\mylabel{item:m=4}
$m=4$, $n\equiv 2\pmod 4$.
\item
\mylabel{item:m=6}
$m=6$, $n\equiv 4\pmod 8$.
\item
\mylabel{item:m=10}
$m=10$, $n\equiv 24\pmod {32}$.
\end{enumerate}
\end{thm}
\begin{proof}
We use the notation of Proposition \ref{prop:gs15} and the paragraph preceding it.

\ref{item:m=3}
Set $n+1=4u$,
$M_1=A\otimes E_2$ and $M_2=P\otimes A$.
Then by Proposition \ref{prop:gs15} \ref{item:4} and Lemma \ref{lem:subfam},
$\{M_1, M_2\}$ is a \hr\ family of order $4$.
Therefore by Lemmas \ref{lem:hr to ans} and \ref{lem:afr add cut} \ref{item:tensor},
 we see that
$T=(E_u\otimes M_1;E_u\otimes M_2;E_{4u})$ is an $(n+1)\times (n+1)\times 3$
\ans\ tensor.
So we see, by Lemma \ref{lem:afr add cut} \ref{item:cut},
 that $T_{\leq n}$ is an $(n+1)\times n\times 3$-tensor 
which satisfies Condition \ref{cond:sp afr}.

Since $(m-1)n-p=2n-(2(n-1)+1)=1$, we see by Corollary \ref{cor:5y}
that $\trank_\RRR(m,n,p)$ contains a number larger than $p$.
So by Remark \ref{rem:more than one}, $\trank_\RRR(m,n,p)$ contains at least two numbers.

\ref{item:m=4}
Set $n+2=4u$, $M_1=A\otimes E_2$, $M_2=P\otimes A$ and $M_3=Q\otimes A$.
Then by Proposition \ref{prop:gs15} \ref{item:4}, we see that 
$\{M_1, M_2, M_3\}$ is a \hr\ family of order $4$.
Therefore, by Lemmas \ref{lem:hr to ans} and \ref{lem:afr add cut} \ref{item:tensor},
 we see that
$T=(E_u\otimes M_1;E_u\otimes M_2;E_u\otimes M_3;E_{4u})$ is an
$(n+2)\times (n+2)\times 4$ \ans\ tensor.
So by Lemma \ref{lem:afr add cut} \ref{item:cut}, $T_{\leq n}$ is \afr.
Moreover, since $(2,1)$-entry of $Q$ is $0$, we see that 
${}^{n<}(E_u\otimes Q\otimes A)_{\leq n}$ is a zero matrix.
Therefore, $T$ satisfies Condition \ref{cond:sp afr}.

Since $(m-1)n-p=3n-(3(n-1)+1)=2$, we see, by  Corollary \ref{cor:5y},
that $\trank_\RRR(m,n,p)$ contains a number larger than $p$.
So by Remark \ref{rem:more than one}, $\trank_\RRR(m,n,p)$ contains at least two numbers.

\ref{item:m=6}
Set $n+4=8u$, $M_1=P\otimes Q\otimes A$, $M_2=A\otimes P\otimes Q$,
$M_3=E_2\otimes A\otimes E_2$, $M_4=E_2\otimes P\otimes A$ and
$M_5=Q\otimes Q\otimes A$.
Then by Proposition \ref{prop:gs15} \ref{item:8} and Lemma \ref{lem:subfam}, 
we see that 
$\{M_1,\ldots, M_5\}$ is a \hr\ family of order $8$.
Therefore, by Lemmas \ref{lem:hr to ans} and \ref{lem:afr add cut} \ref{item:tensor},
$T=(E_u\otimes M_1;\cdots;E_u\otimes M_5;E_{8u})$ is an
$(n+4)\times (n+4)\times 6$ \ans\ tensor.
So by Lemma \ref{lem:afr add cut} \ref{item:cut}, we see that $T_{\leq n}$ is an
$(n+4)\times n\times 6$ \afr\ tensor.
Moreover, since $(2,1)$-entries of $E_2$ and $Q$ are $0$, we see that
$T_{\leq n}$ satisfies Condition \ref{cond:sp afr}.

Since $(m-1)n-p=5n-(5(n-1)+1)=4$, we see, by  Corollary \ref{cor:5y},
that $\trank_\RRR(m,n,p)$ contains a number larger than $p$.
So by Remark \ref{rem:more than one}, $\trank_\RRR(m,n,p)$ contains at least two numbers.

\ref{item:m=10}
Set $n+8=32u$.
By Proposition \ref{prop:gs15} \ref{item:8}, we see that there is a \hr\ family 
$\{L_1, \ldots, L_7\}$ of order $8$.
Since $\{A\}$ is a \hr\ family of order $2$, we see by Theorem \ref{thm:gs16} that
$\{P\otimes A\otimes E_8, Q\otimes E_2\otimes L_1, \ldots, Q\otimes E_2\otimes L_7,
A\otimes E_{16}\}$ is a \hr\ family of order $32$.
Therefore, by Lemmas \ref{lem:hr to ans} and \ref{lem:afr add cut} \ref{item:tensor},
$T=(E_u\otimes A\otimes E_{16}; E_u\otimes P\otimes A\otimes E_8;
E_u\otimes Q\otimes E_2\otimes L_1;\cdots ;E_u\otimes Q\otimes E_2\otimes L_7;E_{32u})$
is an
$(n+8)\times (n+8)\times 10$ \ans\ tensor.
So by Lemma \ref{lem:afr add cut} \ref{item:cut}, we see that $T_{\leq n}$ is an
$(n+8)\times n\times 10$ \afr\ tensor.
Moreover, since $(4,1)$, $(4,2)$ and $(4,3)$-entries of $Q\otimes E_2$ are $0$, we see that
$T_{\leq n}$ satisfies Condition \ref{cond:sp afr}.

Since $(m-1)n-p=9n-(9(n-1)+1)=8$, we see, by  Corollary \ref{cor:5y},
that $\trank_\RRR(m,n,p)$ contains a number larger than $p$.
So by Remark \ref{rem:more than one}, $\trank_\RRR(m,n,p)$ contains at least two numbers.
\end{proof}

\end{document}